\documentclass[12pt,reqno,a4paper]{amsart}
\usepackage{amssymb}
\usepackage{amsfonts}
\usepackage[T1]{fontenc}
\usepackage{enumitem} 
\usepackage{hyperref}\hypersetup{colorlinks=true, citecolor=blue}
\usepackage{color,todonotes}
\usepackage[mathscr]{euscript}
\usepackage{comment} 
\usepackage{esint,amsfonts,amsmath,amssymb,epsfig,mathrsfs,bm,accents,yfonts,mathtools,dsfont,stmaryrd}
\usepackage{graphicx,color}
\usepackage{tikz,fp,ifthen,fullpage,bbm}
\usetikzlibrary{backgrounds}
\usetikzlibrary{decorations.pathmorphing,backgrounds,fit,calc,through}
\usetikzlibrary{arrows}
\usetikzlibrary{shapes,decorations,shadows}
\usetikzlibrary{fadings}
\usetikzlibrary{patterns}
\usetikzlibrary{mindmap}
\usetikzlibrary{decorations.text}
\usetikzlibrary{decorations.shapes}

\theoremstyle{plain} 
\newtheorem{theorem}{Theorem}[section]
\newtheorem{lemma}[theorem]{Lemma}

\newtheorem{proposition}[theorem]{Proposition}
\newtheorem{corollary}[theorem]{Corollary}
\newtheorem{question}[theorem]{Question}
\newtheorem{remark}[theorem]{Remark}
\newtheorem{example}[theorem]{Example}

\newtheorem{definition}[theorem]{Definition}
\numberwithin{equation}{section}
\newtheorem{quest}[theorem]{Question}
\numberwithin{equation}{section}
\newtheoremstyle{mytheorem}
{}
{}
{\it}
{\parindent}
{\bf}
{.}
{ }
{\thmnumber{#2.~}\thmname{#1}\thmnote{~\rm#3}}
\newtheorem{assumption}[theorem]{Assumption}
\newtheoremstyle{myremark}
{}
{}
{\rm}
{\parindent}
{\bf}
{.}
{ }
{\thmnumber{#2.~}\thmname{#1}\thmnote{~\rm#3}}

\newtheoremstyle{myparagraph}
{}
{}
{\rm}
{\parindent}
{\bf}
{.}
{ }
{\thmnumber{#2.~}\thmname{#1}\thmnote{#3}}

\makeatletter
\def\@secnumfont{\sc}
\def\section{\@startsection{section}{1}%
\z@{1.5\linespacing\@plus .2\linespacing}{.7\linespacing}%
{\normalfont\sc\centering}}
\makeatother

\makeatletter
\def\ps@headings{\ps@empty
 \def\@evenhead{%
  \setTrue{runhead}%
  \normalfont\footnotesize
  \rlap{\thepage}\hfil
  \def\thanks{\protect\thanks@warning}%
  \leftmark{}{}\hfil}%
 \def\@oddhead{%
  \setTrue{runhead}%
  \normalfont\footnotesize\hfil
  \def\thanks{\protect\thanks@warning}%
  \rightmark{}{}\hfil \llap{\thepage}}%
\let\@mkboth\markboth}
\makeatother
\makeatletter
\renewenvironment{proof}[1][\proofname]{\par
  \pushQED{\qed}%
  \normalfont \topsep6\p@\@plus6\p@\relax
  \trivlist
  \itemindent\normalparindent
  \item[\hskip\labelsep
    \bfseries
    #1\@addpunct{.}]\ignorespaces
}{%
  \popQED\endtrivlist\@endpefalse
}
\providecommand{\proofname}{Proof}
\makeatother
\newcommand{\Flat}{\mathds{F}}
\newcommand{\Mass}{\mathds{M}}

\newcommand{\R}{\mathbb{R}}

\newcommand{\N}{\mathbb{N}}
\newcommand{\Z}{\mathbb{Z}}

\newcommand{\F}{\mathscr{F}}
\newcommand{\G}{\mathrm{G}}
\newcommand{\Haus}{\mathcal{H}}

\newcommand{\M}{\mathds{M}}

\newcommand{\bd}{\partial}
\newcommand{\eps}{\varepsilon}
\newcommand{\dV}{d_V\kern-1pt}

\DeclareMathOperator{\sign}{sign}

\newcommand{\trait}[3]{\vrule width #1ex height #2ex depth #3ex}

\newcommand{\trace}{\mathchoice%
  {\mathbin{\trait{.12}{1.2}{.03}\trait{.8}{0.09}{0.03}}}
  {\mathbin{\trait{.12}{1.2}{.03}\trait{.8}{0.09}{0.03}}}
  {\mathbin{\hskip.15ex\trait{.09}{.84}{0.02}\trait{.56}{.07}{.02}}\hskip.15ex}
  {\mathbin{\trait{.07}{.6}{.01}\trait{.4}{.06}{.01}}}}

\makeatletter
\@addtoreset{footnote}{section}
\makeatother

\newenvironment{itemizeb}
{\begin{itemize}\itemsep=2pt}{\end{itemize}}
\newcommand{\mass}{\mathds{M}} 
\newcommand{\Po}{\mathscr{P}} 
\newcommand{\Rc}{\mathscr{R}} 
\renewcommand{\flat}{\mathds{F}} 

\newcommand{\modp}{{\rm mod}(p)} 

\newcommand{\Ha}{\mathcal{H}} 

\def\XXint#1#2#3{{\setbox0=\hbox{$#1{#2#3}{\int}$ }
\vcenter{\hbox{$#2#3$ }}\kern-.6\wd0}}

\DeclareFontFamily{U}{matha}{\hyphenchar\font45}
\DeclareFontShape{U}{matha}{m}{n}{
  <-6> matha5 <6-7> matha6 <7-8> matha7
  <8-9> matha8 <9-10> matha9
  <10-12> matha10 <12-> matha12
  }{}

\DeclareSymbolFont{matha}{U}{matha}{m}{n}
\DeclareMathSymbol{\Lt}{3}{matha}{"CE}

\title{On the structure of flat chains with finite mass}
\author{Giovanni Alberti \& Andrea Marchese}

\begin{document}
\begin{abstract}
We prove that every flat chain with finite mass in $\mathbb{R}^d$ with coefficients in a normed abelian group $G$ is the restriction of a normal $G$-current to a Borel set. We deduce a characterization of real flat chains with finite mass in terms of a pointwise relation between the associated measure and vector field. We also deduce that any codimension-one real flat chain with finite mass can be written as an integral of multiplicity-one rectifiable currents, without loss of mass. 

Given a Lipschitz homomorphism $\phi:\tilde G\to G$ between two groups, we then study the associated map $\pi$ between flat chains in $\mathbb{R}^d$ with coefficients in $\tilde G$ and $G$ respectively. In the case $\tilde G=\mathbb{R}$ and $G=\mathbb{S}^1$, we prove that if $\phi$ is surjective, so is the restriction of $\pi$ to the set of flat chains with finite mass of dimension $0$, $1$, $d-1$, $d$. 
\end{abstract}

\maketitle
\section{Introduction}
\label{intro}
Currents with finite mass are not geometrically significant, for instance because for any $k\leq d$ there are non trivial  $k$-dimensional currents in $\R^d$ with finite mass supported on a single point.
The simplest classes of currents which are geometrically significant are normal currents
and flat chains with finite mass. For example, if $\mu_T$ is the measure associated to a $k$-dimensional  
normal current $T$, see \S \ref{s:res-mass}, then $\mu_T(E)=0$ for every set $E$ of vanishing $k$-dimensional integralgeometric measure, and the same holds for 
a $k$-dimensional flat chain with finite mass, see \cite[Theorem 4.2.14]{Federer}.

In fact normal currents and flat-chains with finite mass share many similar properties, 
and therefore it is natural to ask what are the relations between these two notions.
It is immediate to check that if $T$ is a $k$-dimensional normal current, then $\rho T$ is a flat chain with finite mass for every density function $\rho\in L^1(\mu_T)$.
It is also easy to see that the converse of this result is true if $k=d$; namely every $d$-dimensional flat chain of finite mass in $\R^d$ is (represented by) an $L^1$ function.

However, to the best of our knowledge, the converse of this result for arbitrary 
dimension $k$ is not available in the literature; discussing this converse 
and some related questions is the purpose of this paper.\\

Our first result is the following theorem, which is proved in Section \ref{s:proof_main}; the 
necessary definitions are recalled in Section~\ref{currents}. In the sequel, $K$ will denote a convex, compact subset of $\R^d$ with nonempty interior and $G$ is a normed Abelian group.

\begin{theorem}
\label{maincor}
Let $1\le k <d$, and let $T\in\F_k^G(K)$ be such that $\Mass(T)<\infty$. For every $\varepsilon>0$ there exists $T'\in\mathscr{N}_k^G(K)$ and a Borel set $E\subset K$ such that
\begin{itemize}
\item[(i)]
$\bd T'=0$, 
\item[(ii)] $T=T'\trace E$,
\item[(iii)]
$\M(T')\le (2+\eps)\, \M(T)$.
\end{itemize}
In particular, if $G=\Z$ then $T'$ is an
integral current without boundary.
\end{theorem}
Section \ref{s:real} focuses on \emph{real} flat chains with finite mass. For every $1\leq k < d$, we associate to every positive Radon measure $\mu$ on $K\subset\R^d$ a map $V_k(\mu,\cdot):K\to{\rm{Gr}}(\Lambda_k(\R^d))$, whose values are vector subspaces of the space of $k$-vectors, see Definition \ref{d:auxiliary}. Then we prove that such map  characterizes flat chains, in the following sense. 
\begin{theorem}\label{t:char_bundle}
\label{s-6.4}
Let $\mu$ be a positive Radon measure on $K$. Let $\tau\in L^1(\mu)$ be a Borel $m$-vector field.
The following statements are equivalent:
\begin{itemizeb}
\item[(i)]
$\tau(x) \in V_k(\mu,x)$ for $\mu$-a.e.~$x$;
\item[(ii)]
$\tau\mu\in\F_k(K)$.
\end{itemizeb}
\end{theorem}

In \cite{ABM}, Theorem \ref{maincor} and Theorem \ref{t:char_bundle} are used to characterize those multilinear differential operators which are closable and to derive some results concerning the structure of metric currents in the Euclidean space.

We now move our attention to codimension-one real flat chains. We deduce the following result from Theorem \ref{maincor}. See \S \ref{s-measint} for the definition of integral of measures.

\begin{theorem}\label{t:disint}
Let $T\in\F_{d-1}(K)$ with $\mass(T)<\infty$. Then there exist a compact interval $I\subset\R$ and for every $t\in I$ a multiplicity-one rectifiable  current $R_t\in{\mathscr{R}}_{d-1}(K)$ such that the map $t\mapsto R_t$ satisfies properties (a) and (b) of \S \ref{s-measint} and moreover  $$T=\int_{I}R_t dt\quad {\mbox{and}} \quad \Mass(T)=\int_{I}\Mass(R_t) dt.$$ 
\end{theorem}
\begin{remark}
{\rm When $G$ is a finite dimensional vector space, currents with coefficients in $G$ can be defined by duality with certain differential forms, and a $G$-current is uniquely determined by its \emph{components}, which are currents with real coefficients, see \cite[\S 4]{MMST}. Moreover it is easy to see that a $G$-current is a flat $G$-chain if and only if its components are real flat chains. Hence the natural generalization of Theorem \ref{t:char_bundle} holds. Analogously, a $G$-current is rectifiable if and only if its components are rectifiable currents, hence the first part of Theorem \ref{t:disint} remains valid, even though the property $\Mass(T)=\int_{I}\Mass(R_t) dt$ might fail, because typically the mass of a $G$-current is strictly smaller than the sum of the masses of its components.}   
\end{remark}

In Section \ref{s:lifting}, we address the following question. A surjective Lipschitz homomorphism $\phi:\tilde G\to G$ between two groups, induces a surjective map $\pi$ between flat chains in $\mathbb{R}^d$ with coefficients in $\tilde G$ and $G$ respectively. We ask whether or not $\pi$ is surjective between the corresponding sets of flat chains with finite mass. In the case $\tilde G=\mathbb{R}$ and $G=\mathbb{S}^1$, we give a positive answer for $k\in\{0, 1, d-1, d\}$. This allows us to partially extend Theorem \ref{t:char_bundle} and Theorem \ref{t:disint} to other groups than finite dimensional vector spaces.

\subsection*{Acknowledgements}
A part of this paper was conceived when both authors were guests of CIRM Trento through the program Research in pairs. We acknowledge the research center for the support. We would like to thank Benoit Merlet for precious comments. The research of G.A.\ and A.M.\ has been partially supported by the Italian Ministry of University and Research via the PRIN project~2022PJ9EFL "Geometric Measure Theory: Structure of Singular Measures, Regularity Theory and Applications in the Calculus of Variations".

\section{Currents with coefficients in groups}
\label{currents}

In this section we collect the fundamental notions concerning flat chains in $\R^d$ with coefficients in a normed Abelian group. For a thorough discussion about this topic, we refer the reader to the seminal paper \cite{Fleming} and to the more recent \cite{DPH}.

\subsection{Polyhedral chains with coefficients in a normed group}
Let $G = \left(G, + \right)$ denote an Abelian group. A \emph{norm} on $\G$ is any function $\| \cdot \| \colon G \to \R$
satisfying the following properties:
\begin{itemize}
\item[$(i)$] $\| g \| \geq 0$ for every $g \in G$, and $\| g \| = 0$ if and only if $g = 0$;

\item[$(ii)$] $\| - g \| = \| g \|$ for every $g \in G$;

\item[$(iii)$] $\| g + h \| \leq \| g \| + \| h \|$ for every $g,h \in G$.
\end{itemize}

We will assume that $(G,\|\cdot\|)$ is a complete metric space, with respect to the natural metric ${\rm d}(g,h) := \| g - h \|$ for $g,h \in G$.

Let $K$ be a convex, compact subset of $\R^{d}$. A $k$-dimensional \emph{polyhedral chain} in $K$ with coefficients in $G$ (or simply a polyhedral $G$-chain) is a formal \emph{finite} linear combination
\begin{equation} \label{def:polyhedral}
P = \sum_{\ell=1}^{N} g_{\ell} \llbracket\sigma_\ell\rrbracket,
\end{equation}
where $g_{\ell} \in G$ and $\llbracket\sigma_\ell\rrbracket$ are the integral currents associated to the non-overlapping oriented $k$-simplexes $\sigma_\ell\subset K$. A \emph{refinement} of $P$ is any $k$-dimensional polyhedral $G$-chain of the form
\[
\sum_{\ell=1}^{N} \sum_{h=1}^{H_{\ell}} g_{\ell}^h\, \llbracket \sigma_\ell^h \rrbracket,
\]
where $\sigma_\ell^h\cup\ldots\cup \sigma_\ell^{H_\ell}=\sigma_\ell$ and $g_{\ell}^h = g_\ell$ if $\sigma_{\ell}^h$ has the same orientation of $\sigma_\ell$ or $g_\ell^h = - g_\ell$ otherwise. We identify two polyhedral $G$-chains if they have a common refinement. 

The sum of two polyhedral $G$-chains $P_1$ and $P_2$ can be defined by taking refinements such that the corresponding simplexes are either non overlapping or they coincide and then by considering their formal sum, identifying $g_1\llbracket\sigma\rrbracket+g_2\llbracket\sigma\rrbracket$ with $(g_1+g_2)\llbracket\sigma\rrbracket$. With respect to such operation, the set of $k$-dimensional polyhedral $G$-chains in $K$ is a group which will be denoted $\Po_k^G(K)$. Given a polyhedral $G$-chain $P\in\Po_k^G(K)$ as in \eqref{def:polyhedral}, we denote by supp$(P)$ its support, that is the set $\bigcup_{\ell=1}^N\{\sigma_\ell:g_\ell\neq 0\}$ and the
\emph{mass} of $P$ is defined by
\begin{equation} \label{def:mass}
\mass(P) := \sum_{\ell=1}^{N} \| g_\ell \| \Ha^{k}(\sigma_\ell)\,.
\end{equation}

\subsection{Rectifiable currents with coefficients in a normed group}
The group of $k$-dimensional \emph{Lipschitz $G$-chains} in $K$ is given by 
\begin{equation*}
\begin{split}
\mathscr{L}_{k}^G(K) := \bigg\lbrace \sum_{\ell=1}^{N} g_{\ell} \cdot (\gamma_{\ell})_{\sharp} \llbracket\sigma_\ell\rrbracket \, \colon \, &\mbox{each $\sigma_\ell$ is an oriented $k$-simplex in $\R^{k}$,}\\
&\mbox{$g_\ell \in G$ and $\gamma_{\ell} \colon \sigma_\ell \to K$ is Lipschitz}\bigg\rbrace\,,
\end{split}
\end{equation*}
where $\gamma_{\sharp}$ is the \emph{push-forward} according to the Lipschitz map $\gamma$.

The mass functional can be easily extended to $\mathscr{L}_{k}^G(K)$. The $\mass$-completion of $\mathscr{L}_{k}^G(K)$ is the group $\Rc_{k}^G(K)$ of $k$-dimensional \emph{rectifiable currents} with coefficients in $G$. 

\subsection{Boundary and flat norm}\label{subs_flat}
If $P \in \Po_{k}^G(K)$ has the form \eqref{def:polyhedral} then the \emph{boundary} of $P$ is the $(k-1)$-dimensional polyhedral $G$-chain defined by
\begin{equation*}
\partial P := \sum_{\ell=1}^{N} g_{\ell} \partial \llbracket\sigma_{\ell}\rrbracket\,.
\end{equation*}

The \emph{flat norm} of $P$ is the quantity
\begin{equation*}
\flat(P) := \inf\left\lbrace \mass(Q) + \mass(P - \partial Q) \, \colon \, Q \in \Po_{k+1}^G(K) \right\rbrace\,.
\end{equation*}


\subsection{Flat $G$-chains}
The $\flat$-completion of $\Po_{k}^G(K)$ is the group $\F_{k}^G(K)$ of $k$-dimensional \emph{flat $G$-chains} in $K$. The notion of mass extends to $\F_{k}^G(K)$ by relaxation. Moreover, the flat norm of $T\in\F_{k}^G(K)$ can be computed as 
\begin{equation}\label{e:form_flat}
\flat(T) = \inf\left\lbrace \mass(Q) + \mass(F - \partial Q) \, \colon \, Q \in \F^G_{k+1}(K) \right\rbrace\,.
\end{equation}

A flat $G$-chain $T$ of finite mass such that $\partial T$ has also finite mass is called a \emph{normal current} with coefficients in $G$. We denote by $\mathscr{N}_k^G(K)$ the set of $k$-dimensional normal currents in $K$ with coefficients in $G$.\\

There is a close relationship between flat chains and rectifiable currents with coefficients in a group. Being unable to find a reference for the following elementary fact in the setting that we consider, we include the short proof.

\begin{proposition}\label{p:somma}
Let $T\in\F_k^G(K)$. Then there exist $R\in{\mathscr{R}}^G_k(K)$ and $S\in{\mathscr{R}}_{k+1}^G(K)$ such that $T=R+\partial S$. In particular, if $\Mass(T)<\infty$, then $\Mass (\partial S)<\infty$.
\end{proposition}
\begin{proof}
Let $P_n$ be a sequence in $\Po_{k}^G(K)$ such that $\flat(T-P_n)\to 0$ as $n\to\infty$. Without loss of generality we may assume that $\flat(P_{n+1}-P_n)\leq 2^{-n-1}$ for every $n$. For every $n$ let $R_n\in\Po_{k}^G(K)$ and $S_n\in\Po_{k+1}^G(K)$ be such that $$P_{n+1}-P_n=R_n+\partial S_n\quad\mbox{ and }\quad \mass(R_n)+\mass(S_n)\leq 2\flat(P_{n+1}-P_n)\leq 2^{-n}.$$ Note that the sequences $Z_n:=P_1+\sum_{h\leq n}R_h$ and $Q_n:=\sum_{h\leq n}S_h$ converge in mass; we denote the limits $R$ and $S$ respectively and we observe that, being limits in mass of rectifiable currents, they are also rectifiable. Since $P_{n+1}=Z_n+\partial Q_n$ for every $n$, passing the equality in the limit, we conclude that $T=R+\partial S$.
\end{proof}
\begin{remark}\label{r:bound_rect}
{\rm It follows immediately from Proposition \ref{p:somma} that every flat $G$-chain of finite mass is rectifiable provided a \emph{boundary rectifiablity theorem} holds. Namely, whenever all the rectifiable $G$-currents $T\in\mathscr{R}_{k+1}^G(K)$ such that $\Mass(\partial T)<\infty$ also satisfy $\partial T\in \mathscr{R}_{k}^G(K)$, see also \cite{Wh_rec}}.
\end{remark}
The following is a well known characterization of flat chains with finite mass, see \cite[\S 4.1.17]{Federer}. 
\begin{proposition}\label{p:series_normal}
Let $T\in\F_k^G(K)$ with $\Mass(T)<\infty$. Then $T$ is a limit in mass of normal currents.
\end{proposition}
\begin{proof}
If $\M(T)<\infty$, then the current $S$ in the proof of Proposition \ref{p:somma} satisfies $\Mass(\partial S)<\infty$. Moreover the currents $Z_n$ are normal and converge in mass to $R$, hence the currents $Z_n+\partial S$ are normal and they converge in mass to $T$. 
\end{proof}

\subsection{Restriction and mass measure}\label{s:res-mass}
If $T\in\F_{k}^G(K)$ has finite mass and $E\subset K$ is a Borel set, then one can define the flat chain $T\trace E$, see \cite[\S 4]{Fleming}. This allows one to define a Radon measure $\mu_T$ by  
$$\mu_T(E):=\mass(T\trace E)\, .$$
In particular $T\trace E=0$ whenever $\mu_T(E)=0$. For $T\in\Rc_k^G(K)$ of the form $$T=\sum_{\ell\in\N} g_{\ell} \cdot (\gamma_{\ell})_{\sharp} \llbracket\sigma_\ell\rrbracket,$$ the corresponding measure $\mu_T$ is a $k$-rectifiable measure, and more precisely it is absolutely continuous with respect to $\Haus^k\trace(\bigcup_\ell\gamma_{\ell} (\sigma_\ell))$, see \cite[\S 3.6]{DPH}.

\subsection{Currents with coefficients in $\R$ or $\Z$} When $G=\R$, the notion of flat chains and rectifiable currents with coefficients in $G$, and the notions of mass and flat norm defined above correspond to the same notions defined in \cite{Federer}, see \cite[Theorem 4.1.23]{Federer}. In this case, instead of $\Po_k^\R(K),\mathscr{R}_k^\R(K),\F_k^\R(K)$ and $\mathscr{N}^\R_k(K)$ we simply write $\Po_k(K),\mathscr{R}_k(K),\F_k(K)$, and $\mathscr{N}_k(K)$. The same equivalence holds for $G=\Z$. The equivalence between the two notions of mass follows directly from Remark \ref{r:bound_rect} and \cite[Theorem 9.1]{MMST}.

All the elements of these groups are \emph{classical} $k$-dimensional currents, namely continuous linear functionals on the space of smooth differential $k$-forms with compact support. In particular, a $k$-dimensional real flat chain $T$ of finite mass can be identified with a Radon measure with values on the space of $k$-vectors: in this case we write $T=\tau\mu$, being $\mu$ a positive finite measure and $\tau$ a unit $k$-vector field.

\section{Proof of Theorem \ref{maincor}}\label{s:proof_main}
We begin with an elementary approximation lemma.
\begin{lemma}\label{l:uno} Let $T$ be as in Theorem \ref{maincor}. For every $\varepsilon>0$ there exists $P\in\Po_k^G(K)$ with ${\rm{supp}}(P)\subset$ ${\rm{int}}(K)$, $R\in\F_k^G(K)$, and $S\in \F_{k+1}^G(K)$ such that $$T-P=R+\partial S,\quad\mass(R)+\mass(S)\leq\varepsilon,\quad\mass(P)\leq(1+\varepsilon)\mass(T).$$
\end{lemma}
\begin{proof}
By the definition of mass and \eqref{e:form_flat}, for every $\varepsilon$ there exist $P'\in\Po_k^G(K)$, $R'\in\F_k^G(K)$, and $S'\in \F_{k+1}^G(K)$ such that $$T-P'=R'+\partial S',\quad\mass(R')+\mass(S')\leq\varepsilon/2,\quad \mass(P')\leq(1+\varepsilon)\mass(T).$$ We aim to find $P\in\Po_k^G(K)$ with ${\rm{supp}}(P)\subset$ ${\rm{int}}(K)$,  such that 
\begin{equation}\label{emassuno}
\Mass(P)\leq \Mass(P')
\end{equation}
and
\begin{equation}\label{eflatuno}
    \Flat(P-P')\leq\varepsilon/2.
\end{equation} 
Fix $p\in{\rm{int}}(K)$ and for every $\lambda\in[0,1]$, consider the affine map $f_\lambda:K\to K$ defined by $x\mapsto p+\lambda (x-p)$ and observe that, for every $\lambda<1$, $f_\lambda(K)\subset{\rm{int}}(K)$. Notice that from \eqref{def:mass} it follows $\mass((f_\lambda)_\sharp P')=\lambda^k\mass(P')$, hence every choice of  $\lambda\in[0,1)$ yields \eqref{emassuno} for $P:=(f_\lambda)_\sharp P'$.

We show now that \eqref{eflatuno} holds for a suitable choice of $\lambda$. Denoting $h:[0,1]\times K\to K$ the linear homotopy between $f_\lambda$ and $Id$ given by $h(t,x):=(1-t)f_\lambda(x)+tx$, the homotopy formula, see \cite[26.22, 26.23]{Si}, yields
\begin{equation}\label{e:stimaflat1}
    \begin{split}
\flat(P'-P)&=\flat(\partial h_\sharp(\llbracket[0,1]\rrbracket\times P')+h_\sharp(\llbracket[0,1]\rrbracket\times \partial P'))\\
&\leq \Mass(h_\sharp(\llbracket[0,1]\rrbracket\times P'))+\Mass(h_\sharp(\llbracket[0,1]\rrbracket\times \partial P'))\\ 
 &\leq\sup_{{\rm{supp}}(P')}|Id-f_\lambda|(|df_\Lambda|+|dId|)(\mass(P')+\mass(\partial P'))\\
 &\leq(1-\lambda){\rm {diam}}(K)2(\mass(P')+\mass(\partial P')).
    \end{split}
\end{equation}
Hence, \eqref{eflatuno} holds if $\lambda$ is chosen sufficiently close to 1. 
\end{proof}

In the next corollary we observe that the approximating polyhedral current $P$ constructed in Lemma \ref{l:uno} can be found in such a way that $P$ is singular with respect to any fixed Radon measure.
\begin{corollary}\label{c:uno} Let $T$ be as in Theorem \ref{maincor} and let $\mu$ be a Radon measure. For every $\varepsilon>0$ there exist $P\in\Po_k^G(K)$, with $\mu_P\perp\mu$, and $R\in \F_k^G(K)$, $S\in \F_{k+1}^G(K)$ such that $$T-P=R+\partial S,\quad\mass(R)+\mass(S)\leq\varepsilon,\quad \mass(P)\leq(1+\varepsilon)\mass(T).$$ 
\end{corollary}
\begin{proof}
Apply Lemma \ref{l:uno} to obtain $P'\in\Po_k^G(K)$ with ${\rm{supp}}(P')\subset$ ${\rm{int}}(K)$, $R'\in\F_k^G(K)$, and $S'\in \F_{k+1}^G(K)$ such that $$T-P'=R'+\partial S',\quad\mass(R')+\mass(S')\leq\varepsilon/2,\quad\mass(P')\leq(1+\varepsilon/2)\mass(T).$$

For $w\in\R^d$ denote by $\tau_w$ the translation map $x\mapsto x+w$. Since ${\rm{supp}}(P')\subset$ ${\rm{int}}(K)$ there exists $\lambda>0$ such that for every $w\in\mathbb{S}^{d-1}$
it holds ${\rm{supp}}((\tau_{t w})_\sharp P')\subset(K)$
for every $t\in [0,\lambda]$.

We claim that there exists $v\in\mathbb{S}^{d-1}$ such that $\mathcal{H}^k(\tau_{s v}({\rm{supp}}(P))\cap \tau_{t v}({\rm{supp}}(P)))=0$ for every $s\neq t\in[0,\lambda]$. The validity of the claim would complete the proof due to the following facts:
\begin{itemize}
    \item for every $\lambda_0$ in $(0,\lambda]$ and for $t\in[0,\lambda_0]$ denoting $P_t:=(\tau_{t v})_\sharp P'$, the (uncountably many) measures $\mu_{P_t}$ are mutually singular, hence one of them must be singular with respect to $\mu$. We denote $P$ the corresponding current;
    \item $\tau_w$ is an isometry for every $w$ hence $\mass(P)=\mass(P')$;
    \item if $\lambda_0$ is chosen sufficiently small, then $\flat(P-P')\leq \varepsilon/2$, by means of the same argument used to prove \eqref{e:stimaflat1}, replacing $f_\lambda$ with $\tau_{tv}$.
\end{itemize}
Let us focus then on the proof of the claim. We can write $$P = \sum_{\ell=1}^{N} g_{\ell} \llbracket\sigma_{\ell}\rrbracket$$
and since $k<d$ we can choose $v\in\mathbb{S}^{d-1}$ which is not tangent to any of the $\sigma_\ell$'s. This ensures that if $\lambda$ is chosen sufficiently small, then $\Haus^k(\tau_{sv}\sigma_{\ell}\cap\tau_{tv}\sigma_m)=0$ for every $\ell,m\in\{1,\dots,N\}$ and for every $s,t\in[0,\lambda]$. 
\end{proof}

We next apply Corollary \ref{c:uno} to obtain the main ingredient for the proof of Theorem \ref{maincor}, namely the fact that for every flat chain $T$ with finite mass one can find a rectifiable current $R$ which is singular with respect to $T$, have the same boundary and roughly the same mass.

\begin{proposition}\label{main}
Let $T$ and $\mu$ be as in Corollary \ref{c:uno}. For every $\eps>0$ there exists $R\in\mathscr{R}_k^G(K)$ such that \begin{itemize}
\item[(i)]
$\bd R=\bd T$, 
\item[(ii)] $\mu$ and $\mu_{R}$ are mutually singular,
\item[(iii)]
$\M(R)\le (1+\eps)\, \M(T)$.
\end{itemize}
\end{proposition}
\begin{proof}
We can assume $T\neq 0$ and $\varepsilon<1$, and choose a decreasing sequence $\varepsilon_n$ such that
\begin{equation}\label{e:sceltaepsili}
(1+\varepsilon_0)\Mass(T)+\sum_{n\in\N}\varepsilon_n\leq(1+\varepsilon)\Mass(T).    
\end{equation} 
Apply Corollary \ref{c:uno} with $\varepsilon:=\varepsilon_0$ to get $P_0\in\Po_k^G(K)$ with $\mu_{P_0}\perp\mu$, $R_0\in \F_k^G(K)$, and $S_0\in \F_{k+1}^G(K)$ such that 
\begin{equation}\label{e:iter1}
T-P_0=R_0+\partial S_0,\quad\mass(R_0)+\mass(S_0)\leq\varepsilon_0,\quad \mass(P_0)\leq(1+\varepsilon_0)\mass(T).    
\end{equation}
 Applying the boundary operator to the identity in \eqref{e:iter1}, we get 
\begin{equation}\label{e:uno}
\partial T=\partial P_0+\partial R_0\, .
\end{equation}

Now apply again Corollary \ref{c:uno} to $T:=R_0$, with $\varepsilon:=\varepsilon_1$ to get $P_1\in\Po_k^G(K)$, with $\mu_{P_1}\perp\mu$, $R_1\in \F_k^G(K)$, and $S_1\in \F_{k+1}^G(K)$ such that 
\begin{equation}\label{e:iter2}
R_0-P_1=R_1+\partial S_1,\quad\mass(R_1)+\mass(S_1)\leq\varepsilon_1,\quad \mass(P_1)\leq(1+\varepsilon_1)\mass(R_0)\leq(1+\varepsilon_1)\varepsilon_0.    
\end{equation}

Note that, by \eqref{e:uno} and applying the boundary operator to the identity in \eqref{e:iter2}, we have
\begin{equation*}
\partial T=\partial P_0+\partial P_1+\partial R_1\, .
\end{equation*}
Iterating the procedure, we get that $\partial T=\sum_{n\in\N}\partial P_n$. The series converges because the partial sums $\sum_{n=0}^mP_n$ converge in mass to a current $R\in\Rc_k(K;G)$, hence their boundaries converge flat, and the continuity of the boundary operator with respect to the flat norm implies that $\partial R=\partial T$. Moreover, since for every $n\in\N$ we have $\mu_{P_n}\perp\mu$, then also $\mu_{R}\perp\mu$. To obtain the estimate on $\mass(R)$ we observe that $\mass(P_0)\leq(1+\varepsilon_0)\mass(T)$, while for $n\geq 1$ it holds 
$$\mass(P_n)\leq(1+\varepsilon_n)\Mass(R_{n-1})\leq 2\varepsilon_{n-1}\, .$$
By \eqref{e:sceltaepsili}, we get that $\mass (R)\leq (1+\varepsilon)\mass(T)$.
\end{proof}
\begin{remark}{\rm
We actually proved a slightly stronger result, namely that $R$ can be chosen to be a limit in mass of polyhedral chains.}
\end{remark}

\begin{proof}[Proof of Theorem \ref{maincor}]

Let $R$ be obtained by Proposition \ref{main} with $\mu:=\mu_T$. Since $\mu_R\perp\mu_T$, then there exists a Borel set $F$ such that $\mu_T(F)=0$ and $R=R\trace F$. Hence, denoting $E:=\R^d\setminus F$ and $S=T-R$, we have $\partial S=0$ and $T=S\trace E$. The estimate on the mass of $S$ follows from the estimate on the mass of $R$ in Proposition \ref{main}.
\end{proof}

\section{Proof of Theorem \ref{t:char_bundle} and Theorem \ref{t:disint}}\label{s:real}
We begin with the following definition.
\begin{definition}\label{d:auxiliary}
{\rm Let $K\subset\R^d$ be a convex, compact set and let $\mu$ be a Radon measure on $K$. For every $k=1,\dots,d$ and for every point $x$ in the support of $\mu$, 
we denote by $V_k(\mu,x)$ the set 
of all $k$-vectors $v\in\Lambda_m(\R^d)$ for which there exists a $T$ in $\mathscr{N}_k(K)$ with $\bd T=0$ such that
\begin{equation*}
\label{e-tanbund}
\lim_{r\to 0} \frac{ \Mass((T-v\mu)\trace B(x,r)) }{ \mu(B(x,r)) }
= 0
\, .
\end{equation*}
We set $V_k(\mu,x):=\{0\}$
when $x$ does not belong to the support of $\mu$.}
\end{definition}
With the same argument of \cite[Lemma 6.9 ]{AM}, it is possible to prove that $V_k$ is universally measurable and in particular it coincides with a Borel map on a Borel set of full measure.

\begin{remark}
\rm{When $k=1$, the definition above coincides with the definition of \emph{auxiliary bundle} given in \cite[\S6.1]{AM}. This coincides almost everywhere with the so called \emph{decomposability bundle}, defined in terms of all the possibilities to write locally the measure $\mu$ as an integral of 1-rectifiable measures, see \cite[\S 2.6 and Theorem 6.4]{AM}.

It is a simple consequence of Theorem \ref{t:disint} that $V_{d-1}(\mu,x)$ coincides almost everywhere with a variant of the decomposability bundle defined replacing 1-rectifiable measures with $(d-1)$-rectifiable ones.

One could be tempted to conjecture that for every $k$ the bundle $V_k(\mu,x)$ coincides almost everywhere with a variant of the decomposability bundle obtained replacing 1-rectifiable measures with $k$-rectifiable ones. Interestingly, this conjecture is false, due to an example by Schioppa, see \cite{Schio}, motivated by a question of F. Morgan on the possibility to extend Theorem \ref{t:disint} to normal currents of general dimension, see \cite[Problem 3.8, pg. 446]{Open}.} 
\end{remark}

\begin{proof}[Proof of Theorem \ref{t:char_bundle}]
Following verbatim the proof of \cite{AM}, we obtain that (i) is equivalent to the following property, where we denote $K_\delta:=\{x\in\R^d:{\rm dist}(x,K)\leq \delta\}$.
\begin{itemize}
    \item [(ii')] there exists $T\in\mathscr{N}_k(K_\delta)$ such that $T=\tau\mu+\sigma$, where $\sigma$ and $\mu$ are mutually singular. 
\end{itemize}
On the other hand, by Theorem \ref{maincor}, (ii) and (ii') are equivalent.
\end{proof}

It follows immediately from Theorem \ref{t:char_bundle} that if $T\in\mathscr{F}_k(K)$ has finite mass, then $V_k(\mu_T,x)\neq \{0\}$ for $\mu_T$-a.e. $x$. It is not clear if the same holds true for flat chains with coefficients in any group $G$.
\begin{quest}\label{q:bundle}
{\rm Let $T\in\F_k^G(\R^d)$ for an arbitrary group $G$. Is it true that if $\mass(T)<\infty$ then $V_k(\mu_T,x)\neq 0$ for $\mu_T$-a.e. $x$?}
\end{quest}

We now switch our attention to codimension-1 real flat chains with finite mass. It is easy to deduce from Theorem \ref{maincor} that they can be written as integrals of rectifiable currents with integer coefficients, without loss of mass. To this purpose we recall the following definition, see \cite[\S 2.3]{AM}.

\subsection{Integration of measures}
\label{s-measint}
Let $I\subset\R$ be a compact interval and for every $t\in I$ 
let $\mu_t$ be a real- or vector-valued measure on $\R^d$ such that:
\begin{itemizeb}
\item[(a)]
for every Borel set $E$ in $\R^d$ the function $t\mapsto \mu_t(E)$ 
is measurable; 
\item[(b)]
$\int_I \Mass(\mu_t) \, dt <+\infty$.
\end{itemizeb}
Then we denote by $\int_I \mu_t\, dt$ the measure on 
$\R^d$ defined by
\[
{\textstyle \big[ \int_I \mu_t\, dt \big]}(E)
:= \int_I \mu_t(E) \, dt
\quad\text{for every Borel set $E$ in $\R^d$.}
\]

\begin{proof}[Proof of Theorem \ref{t:disint}]
Let $T'\in\mathscr{N}_{d-1}(K)$ and $E\subset K$ be the normal current and the set obtained applying Theorem \ref{maincor} with $\varepsilon=1$. Since $T'=\partial S$ for some $S\in\mathscr{R}_d(K)$ with the additional property that $S$ is represented by a $BV$ function, then \cite[Theorem 1.10]{Alb_rk1} implies that $T'$ admits a decomposition $T'=\int_IR'_t dt$ as in the statement of the theorem. The conclusion follows setting $R_t:=R'_t\trace E$ for every $t\in I$.
\end{proof}

A complete generalization of Theorem \ref{t:disint} for flat chians with coefficients in an arbitrary group $G$ is not feasible, for instance because it is not clear how to understand an integral of rectifiable $G$-currents, given that flat chains with coefficients in a general group $G$ are not defined by duality. Nonetheless, a partial generalization would be represented by an affirmative answer to the following question.
\begin{quest}\label{q:decomp}
{\rm Let $T\in\F_{d-1}^G(\R^d)$ for an arbitrary group $G$. Is it ture that if $\mass(T)<\infty$ then $\mu_T=\int_I\mu_t dt$ for a compact set $I$ and a family of $(d-1)$-rectifiable measures $(\mu_t)_{t\in I}$?}
\end{quest}

Theorem \ref{t:disint} also allows us to shed light on the structure of the map $V_k(\mu,\cdot)$ of Definition \ref{d:auxiliary}. To every $k$-vector $v$ on $\R^d$ one naturally associates a vector subspace ${\rm{span}}(v)\subset\R^d$, see \cite[\S 5.8]{AM}. We observe that in general $V_k(\mu,x)$ does not coincide $\mu$-a.e. with the space of all $k$-vectors whose span is contained in a certain subspace $V(x)\subset\R^d$, by means of the following example.
\begin{example}
{\rm An example announced by A. Mathe, see \cite{Mathe}, shows that there exists a measure $\mu$ on $\R^3$ such that ${\rm{dim}}(V(\mu,x))=2$, for $\mu$-a.e. $x$ and $\mu$ is supported on a purely $2$-unrectifiable set. By Theorem \ref{t:disint}, the latter property implies that $V_2(\mu,x)=0$, for $\mu$-a.e. $x$.

Let $\nu:=\mu_x\times\mu_y$ be a measure on $\R^6\simeq\R_x^3\times\R_y^3$ where $\mu_x$ and $\mu_y$ are two copies of the measure $\mu$ on $\R^3_x$ and $\R^3_y$ respectively. We can write $V(\mu_x,\cdot)={\rm{span}}(v_1,v_2)$ for some unit vector fields $v_1, v_2$ on $\R^3_x$ and $V(\mu_y,\cdot)={\rm{span}}(w_1,w_2)$ for $w_1, w_2$ on $\R^3_y$. With a small abuse of notation we use $v_i$ (resp. $w_i$), $i=1,2$, both for the vector field in $\R^3_x$ (resp. $\R_y^3$) and for the vector field $(v_i,0)$ (resp. $(0,w_i)$) in $\R^6$. 

By Theorem \ref{t:char_bundle} we have that $v_1\mu_x, v_2\mu_x, w_1\mu_y,$ and $w_2\mu_y$ are 1-dimensional real flat chains, so that in particular the cartesian products $(v_1\wedge w_1)\nu, (v_1\wedge w_2)\nu, (v_2\wedge w_1)\nu,$ and $(v_2\wedge w_2)\nu,$ are 2-dimensional real flat chains on $\R^6$, see \cite[\S 4.1.8]{Federer}. Hence by Theorem \ref{t:char_bundle} we have
$$\{v_1\wedge w_1, v_1\wedge w_2, v_2\wedge w_1, v_2\wedge w_2\}\subset N^2(\nu,\cdot).$$
Since for every $v,w$ it holds ${\rm{span}}\{v,w\}={\rm{span}}(v\wedge w)$, then, should $V_2(\nu,x)$ coincide $\nu$-a.e. with the space of all 2-vectors whose span is contained in a certain subspace $V(x)\subset \R^6$, we would have $V(x)\supset {\rm{span}}\{v_1(x),v_2(x),w_1(x),w_2(x)\}$, so that in particular $v_1(x)\wedge v_2(x)\in V_2(\nu,x)$ for $\nu$-a.e. $x$. Again, Theorem \ref{t:char_bundle} would imply that $(v_1\wedge v_2)\nu$ is a real flat chain and so it would be the push forward according to the projection on $\R_x^3$, namely the flat chain $(\pi_x)_\sharp(v_1\wedge v_2)\nu=(v_1\wedge v_2)\mu_x$, thus contradicting the fact that $v_2(\mu,x)=0$ for $\mu$-a.e. $x$.} 
\end{example}

\section{Lifting flat chains with finite mass}\label{s:lifting}
Let $G$ and $\tilde G$ be normed abelian groups and let $\phi:\tilde G\to G$ be a surjective homomorphism such that 
\begin{equation}\label{eq:const_p}
\|\phi(g)\|_G\leq C\|g\|_{\tilde G},
\end{equation}
and 
\begin{equation}\label{eq:const_p2}
\|g\|_G\geq C^{-1} \inf_{\tilde g\in \phi^{-1}(g)}\|\tilde g\|_{\tilde G},
\end{equation}
for some constant $C> 0$.

As a guiding example, one may consider the case $\tilde G=\R$ and $G=\mathbb{S}^1=\R/\Z$, being $\phi$ the projection on the quotient, $\|\cdot\|_{\tilde{G}}$ the Euclidean norm and $\|\cdot\|_{G}$ the quotient norm, i.e. 
\begin{equation}\label{e:normG}
\|\phi(\tilde g)\|_G:=\min\{{\rm{frac}}(\tilde g), 1-{\rm{frac}}(\tilde g)\}, 
\end{equation}
where ${\rm{frac}}(x):=x-\lfloor x \rfloor$ is the fractional part of $x$, and $\lfloor x\rfloor$ is the largest integer not greater than $x$.

For every convex, compact set $K\subset\R^d$ and for every $0\leq k\leq d$, the map $\phi$ induces a natural map $\pi:\Po_{k}^{\tilde G}(K)\to\Po_{k}^G(K)$, defined as follows. Given $P:=\sum_i\tilde g_i\llbracket\sigma_i\rrbracket$ as in \eqref{def:polyhedral} we define 
$$\pi(P):=\sum_i \phi(\tilde g_i)\llbracket\sigma_i\rrbracket.$$ 
Clearly $\pi$ is surjective and commutes with the boundary operator. Moreover by \eqref{eq:const_p} we get $\mass(\pi(P))\leq C\mass(P)$, for every $P\in\Po_{k}^{\tilde G}(K)$. Hence we deduce that $\flat(\pi(P))\leq C\flat(P)$, for every $P\in\Po_{k}^{\tilde G}(K)$. We conclude that $\pi$ extends to a surjective map (still denoted $\pi$) between $\F_k^{\tilde G}(K)$ and $\F_k^G(K)$ which is $C$-Lipschitz with respect to the flat norm $\flat$.\\

In this section we address the following question, which are natural generalizations of classical questions in the framework of flat chains modulo $p$, see \cite[4.2.26]{Federer}, \cite{Whitemod4} \cite{Young}, and \cite[Question 3.5 and Question 3.7]{MS}.

\begin{question}\label{q:lifting}
{\rm Is the restriction of the map $\pi$ surjective between $\{\tilde T\in\F_k^{\tilde G}(K):\mass(\tilde T)<\infty\}$ and $\{T\in\F_k^G(K):\mass(T)<\infty\}$?}
\end{question}
It follows immediately from \eqref{eq:const_p2} that the answer to Question \ref{q:lifting} restricted to the class of rectifiable currents is positive. In particular, we have the following
\begin{corollary}\label{c:rec}
The answer to Question \ref{q:lifting} is positive for $k=d$.
\end{corollary}

Let us briefly explain why it is substantially harder to give a positive answer to Question \ref{q:lifting} in the class of flat chains than in the class of rectifiable currents. Given $T\in\mathscr{F}_k^G(K)$ with $\mass(T)<\infty$ one can apply Proposition \ref{p:somma} to write $T=R+\partial S$ with $R\in\mathscr{R}_k^G(K)$ and $S\in\mathscr{R}_{k+1}^G(K)$ with the additional property that $\mass(\partial S)<\infty$. By \eqref{eq:const_p2} there are $\tilde R\in\mathscr{R}_k^{\tilde G}(K)$ and $\tilde S\in\mathscr{R}_{k+1}^{\tilde G}(K)$ such that $\tilde T:=\tilde R + \partial \tilde S\in\mathscr{F}_k^{\tilde G}(K)$ satisfies $\pi(\tilde T)=T$, but it could happen that $\M(\partial\tilde S)=+\infty$, and hence $\mass(\tilde T)=+\infty$. In conclusion, given a current $S\in\mathscr{R}_{k+1}^G(K)$ with $\mass(\partial S)<\infty$, the challenge is to construct a current $\tilde Z\in\mathscr{R}_{k+1}^{\tilde G}(K)$ with $\mass (\partial \tilde Z)<\infty$ and $\pi(\partial\tilde Z)=\partial S$.\\

The possibility to give a positive answer to Question \ref{q:lifting} for $k<d$ is strictly related to the analogue in our setting of a classical question on the boundedness of the ratio between the minimal filling of a given integral boundary among integral and normal currents, see \cite[Problem 1.13, pg. 443]{Open}. We begin with the following definition.

\begin{definition}\label{d:br}
We say that an homeomorphism $\phi$ as in \eqref{eq:const_p} and \eqref{eq:const_p2} satisfies the \emph{bounded ratio property} $BR(k,D)$ for some $1\leq k\leq d-1$ and $D>0$ if the following holds.
For every $Q\in\Po_k^{\tilde G}(K)$ with $\pi(\partial Q)=0$ there exists $Q'\in\Po_k^{\tilde G}(K)$ with 
$$\pi(Q')=0,\quad \partial Q'=\partial Q,\quad\mbox{and}\quad\mass(Q')\leq D\mass (Q).$$
\end{definition}

We say that $G$ is \emph{locally compact} if for every $M>0$ the set $\{g\in G:\|g\|\leq M\}$ is compact. By \cite[Lemma 7.4]{Fleming}, if $G$ is locally compact, for every sequence $T_n\in\F_k^G(K)$ with $\sup\{\Mass(T_n)+\Mass(\partial T_n)\}<\infty$ there exists $T\in\F_k^G(K)$ and a subsequence $T_{n_i}$ such that $\Flat(T_{n_i}-T)\to 0$.
\begin{proposition}\label{p:main}
Assume that the homeomorphism $\phi$ satisfies \eqref{eq:const_p}, \eqref{eq:const_p2}, and the bounded ratio property $BR(k,D)$ and assume that $G$ is locally compact. Then for every $T\in\F_k^G(K)$ with $\Mass(T)<\infty$ and for every $\varepsilon>0$ there exists $\tilde T\in\F_k^{\tilde G}(K)$ with 
$$\pi(\tilde T)=T\quad\mbox{and}\quad\Mass(\tilde T)\leq C(2+2D)(1+\varepsilon)\Mass(T).$$ In particular the answer to Question \ref{q:lifting} is positive. 
\end{proposition}
\begin{proof}
    Let $P_n$ be a sequence in $\Po_{k}^G(K)$ be such that $\flat(T-P_n)\to 0$ as $n\to\infty$ and 
\begin{equation}\label{e:supmasse}
\sup_n\{\mass(P_n)\}=:M\leq \left(1+\frac\varepsilon 2\right) \Mass(T).
\end{equation}
Without loss of generality we may assume that $$\flat(P_{n+1}-P_n)\leq 2^{-n-1}\frac\varepsilon 4\Mass(T), \quad\mbox{for every $n\in\N$}.$$  Hence there exist $R_n\in\Po_{k}^G(K)$ and $S_n\in\Po_{k+1}^G(K)$ be such that 
\begin{equation}\label{e:usual}
P_{n+1}-P_n=R_n+\partial S_n    
\end{equation}
and 
\begin{equation}\label{e:usual2}
    \mass(R_n)+\mass(S_n)\leq 2\flat(P_{n+1}-P_n)\leq 2^{-n-2}\varepsilon \Mass(T).
\end{equation} 


For every $n\in\N$, denote
$$\hat R_n:=P_0+\sum_{j=0}^{n-1}R_j\quad \hat S_n:=\sum_{j=0}^{n-1}S_j$$
and observe that by \eqref{e:usual} we have $P_n=\hat R_n+\partial\hat S_n$ and by \eqref{e:supmasse} and \eqref{e:usual2}
\begin{equation*}\label{e:masstilder}
    \Mass(\hat R_n)\leq\Mass(P_0)+\frac\varepsilon 2\Mass(T)\leq (1+\varepsilon)\Mass(T)
\end{equation*}
so that
\begin{equation}\label{e:masstildesn}
    \Mass(\partial\hat S_n)\leq\Mass(P_n)+\Mass(\hat R_n)\leq(2+2\varepsilon)\Mass(T).
\end{equation}
Let $\tilde P_0$ and $\tilde R_n\in\Po_{k}^{\tilde G}(K)$ be the polyhedral $\tilde G$-chains defined replacing each multiplicity $g$ in $P_0$ and in $R_n$ respectively with an element $\tilde g$ such that $\phi(\tilde g)=g$ and $\|\tilde g\|\leq C \|g\|$, which exists by \eqref{eq:const_p2}. 
Observe that by \eqref{e:supmasse} we have $\Mass(\tilde P_0)\leq C(1+\varepsilon/2)\Mass(T)$ and moreover by \eqref{e:usual2}
$$\mass(\tilde R_n)\leq C\mass(R_n)\leq C2^{-n-2}\varepsilon \Mass(T).$$ 
Hence the sequence $\tilde Z_n:=\tilde P_0+\sum_{i=0}^{n-1} \tilde R_i$ is a Cauchy sequence (in mass) so it converges to a flat chain $\tilde Z$. Moreover 
\begin{equation}\label{e:zetan}
    \pi\tilde Z_n=\hat R_n\quad\mbox{and}\quad\mass(\tilde Z)\leq C(1+\varepsilon)\Mass(T).
\end{equation}
Similarly let $\tilde X_n\in\Po_{k}^{\tilde G}(K)$ be defined replacing each multiplicity $g$ in $\partial \hat S_n$ with $\tilde g$ as above, so that 
\begin{equation}\label{e:masssn}
\Mass(\tilde X_n)\leq C\Mass(\partial \hat S_n)    
\end{equation}
Observe that $\tilde X_n$ might fail to be a boundary, but we still have $$\pi(\partial \tilde X_n)=\partial(\pi (\tilde X_n))=\partial(\partial \hat S_n)=0.$$
Hence by Definition \ref{d:br} and combining \eqref{e:masstildesn} and \eqref{e:masssn}, there exists $\tilde Y_n\in\Po_{k}^{\tilde G}(K)$ with 
\begin{equation}\label{e:basta}
\pi(\tilde Y_n)=0, \quad\partial \tilde Y_n=\partial\tilde X_n,\quad\mbox{and}\quad \Mass(\tilde Y_n)\leq D\Mass(\tilde X_n)\leq CD(2+2\varepsilon)\Mass(T).    
\end{equation}

Setting $W_n:=\tilde X_n-\tilde Y_n$ and combining \eqref{e:masstildesn}, \eqref{e:masssn} and \eqref{e:basta} we obtain 
\begin{equation}\label{e:wn}
    \partial W_n=0,\quad\pi(W_n)=\partial\hat S_n\quad\mbox{and}\quad \Mass(W_n)\leq C(D+1)(2+2\varepsilon)\Mass(T).
\end{equation}

Since $G$ is locally compact, there exists $W\in\F_k^{\tilde G}(K)$ such that, up to subsequences, $\Flat(W_n-W)\to 0$ and 
\begin{equation}\label{e:basta2}
    \Mass(W)\leq C(D+1)(2+2\varepsilon)\Mass(T).
\end{equation}
Eventually, setting $\tilde T:=\tilde Z+\tilde W$ and combining \eqref{e:zetan} and \eqref{e:wn}, the continuity of $\pi$ wrt the flat norm and the fact that $P_n=\hat R_n+\partial\hat S_n$ implies that $\pi(\tilde T)=T$ and moreover the estimates in \eqref{e:zetan} and \eqref{e:basta2} yield $$\Mass(\tilde T)\leq C(2+2D)(1+\varepsilon)\Mass(T).$$
\end{proof}

In the sequel we collect some partial answers to Question \ref{q:lifting}, which allow to answer positively to Question \ref{q:bundle} and to Question \ref{q:decomp} in some cases of interest.
We restrict to the guiding example mentioned above, namely we make the following \\



\begin{assumption}\label{assumption}
{\rm Let $\tilde G:=\R$, $G:=\mathbb{S}^1=\R/\Z$, and $\phi:\R\to\R/\Z$ is the projection on the quotient. We let $\|\cdot\|_{\tilde G}$ be the Euclidean norm on $\R$ and $\|\cdot\|_{G}$ the induced norm on $\mathbb{S}^1$, see \eqref{e:normG}.}
\end{assumption}
\begin{theorem}\label{t:lifting0dim}
    Let $G$ and $\tilde G$ be as in Assumption \ref{assumption} and let $k=0$. Then the answer to Question \ref{q:lifting} is positive.
\end{theorem}
\begin{proof}
    We can argue as in the proof of Proposition \ref{p:main} observing that we do not need the property of Definition \ref{d:br}. Indeed we do not need the chains $\tilde Y_n$ in \eqref{e:basta} and we can take $W_n:=\tilde X_n$, because the latter sequence is already $\Flat$-precompact. 
\end{proof}

The next lemma provides the main tool to answer Question \ref{q:lifting} for $k=d-1$. The proof uses a strategy already employed in \cite{DavIgn, Merl}.

\begin{lemma}\label{l:lifting_d-1}
Let $G$ and $\tilde G$ be as in Assumption \ref{assumption}. Let $P\in\Po_{d}^G(K)$. Then there exists $\tilde P\in\Po_{d}^{\tilde G}(K)$ such that 
$$\pi(\tilde P)=P, \quad\mass(\tilde P)\leq 3 \mass(P), \quad {\mbox and }\quad \mass(\partial\tilde P)\leq 5 \mass(\partial P).$$
\end{lemma}

\begin{remark}
{\rm We note that Lemma \ref{l:lifting_d-1} implies the validity of Statement $\mathcal{P}_m$ of \cite[\S 4.1]{MS} for $m=n$ with a constant $C$ which is independent of $p$. This can be obtained simply dividing by $p$ the multiplicities of the current $P$ in such statement, so that the $\modp$ equivalence relation on $\Z$ turns into the $\mod(1)$ equivalent relation on $\Z/p\subset\R$. In turn, Statement $\mathcal{S}_m$ of \cite[\S 4.1]{MS} follows from Statement $\mathcal{P}_m$ with the same constant, see \cite[Remark 4.2]{MS} and the two statements imply positive answers to \cite[Question 3.5 and Question 3.7]{MS}. 
}
\end{remark}

\begin{proof}[Proof of Lemma \ref{l:lifting_d-1}]
Write $P$ in the form $P:=\sum_i g_i\llbracket\sigma_i\rrbracket$, for $g_i\in\mathbb{S}^1$ as in \eqref{def:polyhedral} the $\sigma_i$ being endowed with the standard orientation of $\R^d$. Accordingly we can write $\partial P:=\sum_j (g_j^+-g_j^-) \llbracket S_j\rrbracket$, where $S_j$ are $(d-1)$-simplexes and $g_j^\pm$ is the multiplicity of the $d$-simplex $\sigma_j^\pm$ appearing in the decomposition of $P$ having $S_j$ in the boundary and such that the positive orientation on $\sigma_j^\pm$ induces on $S_j$ a positive (resp. negative) orientation (we simply set $g_j^\pm=0$ if there is no such $d$-simplex $\sigma_j^\pm$). Let us now fix $\theta\in(\frac{1}{4},\frac{3}{4})$ to be chosen later. For every $g\in[0,1)$ let us denote
\begin{equation}\label{e:g_tilde}
    \tilde g=\tilde g(\theta):=
    \begin{cases}
      g, & \text{if}\ g\leq \theta; \\
      g-1, & \text{otherwise.}
    \end{cases}
  \end{equation}
Consider now the polyhedral $\tilde G$-chain $\tilde P:=\sum_i \tilde g_i\llbracket\sigma_i\rrbracket$. Obviously $\pi(\tilde P)=P$. Since $\|\tilde g\|_{\tilde G}\leq 3\|g\|_G$ for every $g\in G$, then we have
$$\mass(\tilde P)\leq 3\mass(P).$$
Moreover, denoting $S:=\bigcup_j S_j$ we have 
$$\mass(\partial P)=\int_S \|g^+-g^-\|_G d\Haus^{d-1}=\sum_j \|g_j^+-g_j^-\|_G \Haus^{d-1}(S_j)$$
and similarly
$$\mass(\partial \tilde P)=\int_S \|\tilde g^+-\tilde g^-\|_{\tilde G} d\Haus^{d-1}=\sum_j \|\tilde g_j^+-\tilde g_j^-\|_{\tilde G} \Haus^{d-1}(S_j).$$
Recalling that the definition of $\tilde g$ (and henceforth of $\tilde P$) depends on $\theta$, we encode this dependence in the notations $\tilde g_\theta^\pm$ and $\tilde P_\theta$. Integrating in $\theta\in[\frac{1}{4},\frac{3}{4}]$, we get
$$\int_{\frac{1}{4}}^{\frac{3}{4}}\mass(\partial \tilde P_\theta)d\theta=\int_{\frac{1}{4}}^{\frac{3}{4}}\int_S \|\tilde g^+_\theta-\tilde g^-_\theta\|_{\tilde G} d\Haus^{d-1} d\theta=\int_S\int_{\frac{1}{4}}^{\frac{3}{4}} \|\tilde g^+_\theta-\tilde g^-_\theta\|_{\tilde G} d\theta d\Haus^{d-1}.$$
Observe that if $\theta$ does not belong to the open segment with extremes $g^-$ and  $g^+$ (whose length is $\|g^+-g^-\|_G$), we have 
$$\|\tilde g^+_\theta-\tilde g^-_\theta\|_{\tilde G}=\|g^+-g^-\|_{\tilde G}\leq 3 \|g^+-g^-\|_G,$$
and otherwise we have trivially
$$\|\tilde g^+_\theta-\tilde g^-_\theta\|_{\tilde G}\leq 1.$$
Hence we deduce that 
$$\int_{\frac{1}{4}}^{\frac{3}{4}}\mass(\partial \tilde P_\theta)d\theta\leq \int_S\int_{\frac{1}{4}}^{\frac{3}{4}} 3\|g^+-g^-\|_G d\theta d\Haus^{d-1} + \int_S \|g^+-g^-\|_G\cdot 1 d\Haus^{d-1}=\frac{5}{2}\mass(\partial P),$$
which implies that there exists $\theta\in(\frac{1}{4},\frac{3}{4})$ such that $\mass(\partial\tilde P_\theta)\leq 5 \mass(\partial P)$.
\end{proof}

\begin{theorem}\label{t:lifting_d-1}
Let $G$ and $\tilde G$ be as in Assumption \ref{assumption} and let $k=d-1$. Then the answer to Question \ref{q:lifting} is positive.
\end{theorem}

\begin{proof}
We aim at proving the validity of the bounded ratio property $BR(d-1,D)$ of Definition \ref{d:br} with $D=6$, which suffices to prove Theorem \ref{t:lifting_d-1} thanks to Proposition \ref{p:main}. Consider $Q\in\Po_{d-1}^{\tilde G}(K)$ with $\pi(\partial Q)=0$. By the cone construction, we can write $\pi(Q)=\partial S$ for some $S\in\Po_{d}^{G}(K)$ with $$\Mass(S)\leq C(K)\Mass(\pi(Q))\leq C(K)\Mass(Q).$$ By Lemma \ref{l:lifting_d-1} there exist $\tilde S\in\Po_{d}^{\tilde G}(K)$ with 
$$\pi(\tilde S)=S,\quad\mass(\tilde S)\leq 3 \mass(S)\leq 3C(K)\mass(Q), \quad {\mbox and }\quad \mass(\partial\tilde S)\leq 5 \mass(\partial S)\leq 5\mass(Q).$$
Now set $Q':=Q-\partial \tilde S$. We have $\partial Q'=\partial Q$,
$$\pi(Q')=\pi(Q)-\pi(\partial\tilde S)=\pi(Q)-\partial S=0,$$
and 
$$\mass(Q')\leq \Mass(Q)+\Mass(\partial\tilde S)=6\Mass(Q).$$
\end{proof}

\begin{corollary}\label{c:disint_s1}
Let $G$ be as in Assumption \ref{assumption}. Then the answer to Question \ref{q:decomp} is affirmative. In particular the answer to Question \ref{q:bundle} is affirmative for $k=d-1$.
\end{corollary}
\begin{proof}
Combining Theorem \ref{t:lifting_d-1} and Theorem \ref{t:disint} we infer the existence of a real flat chain with finite mass $\tilde T\in\mathscr{F}_{d-1}(K)$ such that $\pi(\tilde T)=T$ and a map $t\mapsto R_t$ as in \S \ref{s-measint} taking values in the space $\mathscr{R}_{d-1}(K)$ such that $\tilde T=\int_I R_t dt$ and $\mass(\tilde T)=\int_I \mass(R_t) dt$. Since $\mu_T$ is absolutely continuous w.r.t $\mu_{\tilde T}$ and for every $t\in I$ the measure $\mu_{\pi(R_t)}$ is absolutely continuous w.r.t. $\mu_{R_t}$, we deduce the desired conclusion.
\end{proof}


\begin{theorem}\label{t:lifting_0}
Let $G$ and $\tilde G$ be as in Assumption \ref{assumption} and let $k=1$. Then the answer to Question \ref{q:lifting} is positive.
\end{theorem}
\begin{proof}
    We aim at proving the validity of the bounded ratio property $BR(1,D)$ of Definition \ref{d:br} with $D=1$, which suffices to prove Theorem \ref{t:lifting_d-1} thanks to Proposition \ref{p:main}.
For every $P\in\Po_1^{\tilde G}$ as in \eqref{def:polyhedral}, let 
\begin{equation*}
\mathcal{I}_P:=\{\ell\in\{1,\dots,N\}: \phi(g_\ell)\neq 0_G\}    
\end{equation*}
and denote $\hat P:=\sum_{\ell\in\mathcal{I}_P}g_\ell\llbracket\sigma_\ell\rrbracket$. Observe that, since $\pi(\partial(Q-\hat Q))=0$, then $\pi(\partial \hat Q)=0$. We claim that there exists $Z\in\Po_1^{\tilde{G}}(K)$ such that 
$$\partial Z=0,\quad\mass(\hat Q-Z)\leq\mass(\hat Q),\quad \mbox{and}\quad \mathcal{I}_{\hat Q-Z}\subsetneqq\mathcal{I}_Q.$$ The validity of the claim implies the validity of the bounded ratio property $BR(1,1)$: indeed, iterating at most $N$ times, we find a current $Z'\in\Po_1^{\tilde{G}}(K)$ such that 
$$\partial Z'=0,\quad \mass(\hat {Q}-Z')\leq\mass(\hat {Q}),\quad\mbox{and}\mathcal{I}_{\hat {Q}-Z'}=\emptyset,$$ so that the current 
$$Q':=Q+Z'=(Q-\hat Q)+(\hat Q-Z')$$ 
yields $$\pi(Q')=0,\quad \partial Q'=\partial Q,\quad\mbox{and}\quad \mass(Q')\leq \mass (Q).$$ 

The proof of the claim is analogous to that of \cite[Lemma A.1]{DLHMS}. Firstly we notice that the support of $\hat Q$ contains a loop, namely a chain of segments $\sigma_{\ell_1}, \dots, \sigma_{\ell_m}$ such that, up to a possible change of orientation, the second endpoint of $\sigma_{\ell_i}$ coincides with the first endpoint of $\sigma_{\ell_{i+1}}$, for every $i=1,\dots,m-1$ and the second endpoint of $\sigma_{\ell_m}$ coincides with the first endpoint of $\sigma_{\ell_1}$. This is an immediate consequence of the fact that the multiplicity $\theta$ of every point in the boundary of $\hat Q$ satisfies $\phi(\theta)=0_G$, but for every $\ell\in\mathcal{I}_P$ we have $\phi(g_\ell)\neq 0_G$. We implicitly switch the sign of $g_{\ell_i}$ whenever in the loop described above there is change in the orientation of $\sigma_{\ell_i}$.
Denote
$$\theta_+:=\min_{i=1,\dots,m}\{\min\{g_{\ell_i}-g:g\in \phi^{-1}(0), g<g_{\ell_i}\}\};$$
$$\theta_-:=\min_{i=1,\dots,m}\{\min\{g-g_{\ell_i}:g\in \phi^{-1}(0), g>g_{\ell_i}\}\}.$$
Define $Z_\pm:=\pm \theta_{\pm}\sum_{i=1}^m\llbracket\sigma_{\ell_i}\rrbracket$. Observe that $\partial Z_\pm=0$ and that $\mathcal{I}_{\hat Q- Z_\pm}\subsetneqq\mathcal{I}_Q$. Moreover the sign of the multiplicity of $\hat Q- Z_\pm$ on each segment $\sigma_{\ell_i}$ is either equal to the sign of $g_{\ell_i}$ or zero, and this implies that 
$$\mass(\hat Q- Z_\pm)=\mass(\hat Q)\mp \theta_\pm\sum_{i=1}^m\Haus^1(\sigma_{\ell_i})\sign(g_{\ell_i}),$$
hence either $Z:=Z_+$ or $Z:=Z_-$ has the desired property.
\end{proof}

%
%
\bibliographystyle{plain}


%
%

\vspace{1cm}
{\parindent = 0 pt\begin{footnotesize}
G.A.
\\
Dipartimento di Matematica\\
Universit\`a di Pisa\\
Largo Bruno Pontecorvo, 5, 56127 Pisa (PI) - Italy\\
e-mail: {\tt giovanni.alberti@unipi.it}
\\

A.M.
\\
Dipartimento di Matematica\\
Universit\`a di Trento\\
Via Sommarive 14, 38123 Povo (TN) - Italy\\
e-mail: {\tt andrea.marchese@unitn.it}

\end{footnotesize}
}
\end{document}